\documentclass{amsart}
\usepackage{amsmath,amsthm,amssymb}

\newcommand{\pair}[1]{\langle #1 \rangle}
\newcommand{\U}{\mathcal{U}}
\newcommand{\V}{\mathcal{V}}
\newcommand{\unifU}{\mathfrak{U}}
\newcommand{\unifpow}[1]{\prod_u\sp{\omega}{#1}}
\newcommand{\supp}[1]{\mathrm{supp}(#1)}
\newcommand{\cl}[2][X]{\mathrm{cl}_{#1}\!\left(#2\right)}
\newcommand{\cov}{\mathbf{C}}
\newcommand{\Star}[1]{\mathrm{St}(#1)}
\newcommand{\Wgame}[1]{\mathrm{Con}(#1)}
\newcommand{\R}{\mathbb{R}}
\newcommand{\model}{\mathcal{M}}
\newcommand{\tpgame}[1]{\mathrm{TP}(#1)}

\newtheoremstyle{theorem}
     {11pt}
     {11pt}
     {}
     {}
     {\bfseries}
     {}
     {.5em}
     {\noindent\thmnumber{#2}. \thmname{#1}\thmnote{#3}}

\theoremstyle{theorem}

\newtheorem{lemma}{Lemma}[section]
\newtheorem{propo}[lemma]{Proposition}
\newtheorem{coro}[lemma]{Corollary}
\newtheorem{ex}[lemma]{Example}
\newtheorem{thm}[lemma]{Theorem}
\newtheorem{ques}[lemma]{Question}

\title{Uniform powers of compacta and the proximal game}
\author[R. Hern\'andez-Guti\'errez]{Rodrigo Hern\'andez-Guti\'errez}
\author[P. J. Szeptycki]{Paul J. Szeptycki}

\address{Department of Mathematics and Statistics, York University, Toronto, ON M3J 1P3, Canada}
\email[R. Hern\'andez-Guti\'errez]{rodhdz@yorku.ca}
\email[P. J. Szeptycki]{szeptyck@yorku.ca}

\date{\today}

\begin{document}

\begin{abstract}
The countable uniform power (or uniform box product) of a uniform space $X$ is a special topology on ${}\sp{\omega}X$ that lies between the Tychonoff topology and the box topology. We solve an open problem posed by P. Nyikos showing that if $X$ is a compact proximal space then the countable uniform power of $X$ is also proximal (although it is not compact). By recent results of J. R. Bell and G. Gruenhage this implies that the countable uniform power of a Corson compactum is collectionwise normal, countably paracompact and Fr\'echet-Urysohn. We also give some results about first countability, realcompactness in countable uniform powers of compact spaces and explore questions by P. Nyikos about semi-proximal spaces.
\end{abstract}

\maketitle

\section{Introduction}

All spaces discussed in this paper will be assumed to be uniformizable (equivalently, Tychonoff, see \cite[Theorem 8.1.20]{eng}).

Let $\pair{X,\unifU}$ be a uniform space. For each $U\in\unifU$, let
$$
\widetilde{U}=\{\pair{f,g}:f,g\in{}\sp{\omega}X\textrm{ and }\forall n<\omega\ (\pair{f(n),g(n)}\in U)\}.
$$
Let $\widetilde{\unifU}$ be the uniformity on ${}\sp{\omega}X$ with base $\{\widetilde{U}:U\in\unifU\}$. Let $\unifpow{\pair{X,\unifU}}$ denote the topological space ${}\sp{\omega}X$ with the topology generated by $\widetilde{\unifU}$. This space was introduced in \cite{bell-unifbox} and called the countable uniform box product of $X$. We will call $\unifpow{\pair{X,\unifU}}$ the countable uniform power of $\pair{X,\unifU}$. If $X$ has only one compatible uniformity, as is the case when $X$ is compact \cite[Theorem 8.3.12]{eng}, we will only write $\unifpow{X}$.

Apparently the countable uniform power was discussed by Scott Williams during the 9th Prague International Topological Symposium (2001), where he asked whether the countable uniform power of a compact space is normal. We also remark that the uniformity $\widetilde{\unifU}$ is also described in \cite[p. 440]{eng} and called the uniformity of uniform convergence (although it is in fact defined for more general products).

In \cite{bell-unifbox}, Bell showed that if $X$ is any Fort space (the one-point compactification of a discrete space), then $\unifpow{X}$ is collectionwise normal and countably paracompact. Later, in \cite{bell-proximal} she proved the corresponding results when $X$ is the $\omega$-power of a Fort space. In this last paper, Bell also defined the class of proximal spaces that encompasses both of these proofs. It turns out that all metric spaces are proximal and in fact proximality is preserved under subspaces, countable products and $\Sigma$-products (all proved in \cite{bell-proximal}). The main feature of proximal spaces is that they have some nice properties.

\begin{thm}\cite{bell-proximal}
If $X$ is a proximal space, then $X$ is collectionwise normal, countably paracompact and Fr\'echet-Urysohn.
\end{thm}

So in fact, in \cite{bell-unifbox} and \cite{bell-proximal} it is just shown that if $X$ is either a Fort space or the $\omega$-power of a Fort space, then $\unifpow{X}$ is proximal. P. Nyikos has shown \cite[Example 2.4]{nyikos-proximal_semiproximal} that there is a uniformity $\unifU$ on $X=\omega\times(\omega+1)$ such that $\unifpow{\pair{X,\unifU}}$ is not Fr\'echet-Urysohn. Thus, in the following we will restrict to compact spaces.

Notice that since $X$ can be embedded as a closed subspace of $\unifpow{X}$, then proximality of $\unifpow{X}$ implies proximality of $X$. Problem 10.3 in \cite{bell-proximal} asks if it is possible to prove that $\unifpow{X}$ is proximal whenever $X$ is proximal. In this paper we answer this Question in the affirmative (for the class of compacta).

\begin{thm}\label{main}
If $X$ is a proximal compactum, then $\unifpow{X}$ is proximal.
\end{thm}

After this, in section \ref{sectfirstcount} we give a characterization of first countability of uniform powers.

In section \ref{sectomega1} we give some observations on $\unifpow{(\omega_1+1)}$, we still don't know whether this space is normal (notice that $\omega_1+1$ is not proximal because it is not Fr\'echet-Urysohn). Thus, the general question still remains open.

\begin{ques} (S. Williams)
Is $\unifpow{X}$ normal whenever $X$ is compact?
\end{ques}

In section \ref{sectprod} we consider semi-proximal spaces, a class of spaces defined by Nyikos in \cite{nyikos-proximal_semiproximal}. Being a weakening of proximality, it is natural to wonder which properties of proximal spaces are also held by semi-proximal spaces. In particular, Nyikos has asked whether semi-proximal spaces are normal (\cite[Problem 13]{nyikos-proximal_semiproximal}). We explore this question for products of subspaces of $\omega_1$.

Finally, in section \ref{topproxgame}, we consider a variation of the proximal game, defined for topological spaces, which we call the \emph{topological proximal game}. We compare the topological proximal game with the proximal game and show that they are not equivalent.

\section{Preliminaries}

One of the motivations for the definition of the uniform product is that its topology lies between the Tychonoff topology and the box topology. Many topological properties that are known to be preserved by Tychonoff products are completely lost for box products. 

Consider ${}\sp{\omega}{(\omega+1)}$. With the Tychonoff topology, this product is homeomorphic to the Cantor set so it is metrizable and compact. However, with the box topology this product is not even Fr\'echet and it is unknown if it is normal in ZFC. Thus, the uniform product presents an intermediate topology that has just begun to be studied.

In the case of a compact metric space $X$, $\unifpow{X}$ is homeomorphic to the space of (continuous) functions from $\omega$ to $X$ with the topology of uniform convergence. It is known that this topology is induced by a metric. Thus, it is natural to try to obtain results for properties that metric spaces have.

Let us give some definitions and results we will need. The unit interval $[0,1]$ will be denoted as $I$. A space $X$ is $\omega$-monolithic if each separable subset of $X$ is metrizable. 

Let $\pair{X,\unifU}$ be a uniform space. Elements of $\unifU$ are called entourages. Entourages are symmetric: if $U\in\unifU$ and $\pair{x,y}\in U$ then $\pair{y,x}\in U$. Given $U\in\unifU$ and $x\in X$, we denote $U[x]=\{y\in X:\pair{x,y}\in U\}$.

A space $X$ is \emph{proximal} if there is a compatible uniformity $\unifU$ on $X$ such that in the following two-player game (called the \emph{proximal game}) there is a winning strategy for player 1. In inning $0$, player 1 chooses an entourage $U_0$ and player 2 chooses $x_0\in X$. In inning $n+1$, player 1 chooses an entourage $U_{n+1}\subset U_n$ and player 1 chooses $x_{n+1}\in U_n[x_n]$. Then player 1 wins the game if either $\bigcap\{U_n[x_n]:n<\omega\}=\emptyset$ or the sequence $\{x_n:n<\omega\}$ converges.

Another game we will use is the following game defined by Gruenhage. Given a space $X$ and $H\subset X$, the \emph{$W$-convergence game} $\Wgame{X,H}$ is played by two players as follows. In inning $n$, player 1 chooses an open set $U_n\subset U_{n-1}$ with $H\subset U_n$ ($U_{-1}=X$) and player 2 chooses a point $p_n\in U_n$. Then player 1 wins the game if the sequence $\{p_n:n<\omega\}$ converges to $H$ (for every neighborhood $W$ of $H$ there is $N<\omega$ such that $x_n\in W$ for all $n>N$). A space $X$ is a $W$-space if for every $x\in X$, player 2 has a winning strategy in $\Wgame{X,\{x\}}$. According to \cite[Lemma 5.1]{bell-proximal}, proximal spaces are in fact $W$-spaces.

Let $\kappa$ be an uncountable cardinal. If $x\in{}\sp\kappa{I}$, define $\supp{x}=\{\alpha<\kappa:x(\alpha)\neq 0\}$, this set is called the support of $x$. Then the $\Sigma$-product of $\kappa$ many copies of the unit interval is the set $\Sigma{{}\sp\kappa{I}}=\{x\in{}\sp\kappa{I}:|\supp{x}|\leq\omega\}$. Recall that a Corson compactum is a compactum contained in a $\Sigma$-product of copies of the unit interval.

\begin{thm}\cite{clontz-gruen-proximalcorson}\label{proximalCorson}
Assume that $X$ is a compactum. Then $X$ is proximal if and only if $X$ is a Corson compactum.
\end{thm}

In order to simplify some proofs, since we are restricting to compacta, we can define a basis for the topology of the uniform power using finite open covers. If $\U$ is an open cover of a space $X$ and $A\subset X$, then the star of $\U$ around $A$ is the set $\Star{A,\U}=\bigcup\{U\in\U:A\cap U\neq\emptyset\}$. A cover $\V$ star-refines a cover $\U$ if $\{\Star{V,\V}:V\in\V\}$ refines $\U$. We will write $\Star{x,\V}$ for $\Star{\{x\},\U}$ when $x\in X$. The following is \cite[Proposition 8.1.6]{eng}.

\begin{lemma}\label{covers}
Let $X$ be any space. Assume that $\cov$ is a set of open covers of $X$ such that
\begin{itemize}
\item[(a)] for every $\U_0,\U_1\in\cov$ there is $\V\in\cov$ that refines both $\U_0$ and $\U_1$;
\item[(b)] for every $\U\in\cov$ there is $\V\in\cov$ that star-refines $\U$; and
\item[(c)] for every $x,y\in X$ there is $\U\in\cov$ such that no member of $\U$ contains both $x$ and $y$.
\end{itemize}
Then the collection of all sets of the form $\bigcup\{U\times U:U\in\U\}$, where $\U\in\cov$ is a base for a uniformity in $X$. 
\end{lemma}

In the case that $X$ is compact, we might restrict $\cov$ to consist of all finite open covers of $X$ and we will say that $\cov$ is basic (for $X$). The uniformity thus defined is the unique uniformity on $X$. Given $\U$ a finite cover of $X$ and $f\in\unifpow{X}$, let
$$
\U[f]=\prod\{\Star{f(i),\U}:i<\omega\},
$$
which is an open neighborhood of $f$ and $\{\U[f]:\U\in\cov\}$ is a local base at $f$.

\section{Proof of Theorem \ref{main}}

Let $X$ be a proximal compactum. By Theorem \ref{proximalCorson}, there is some $\kappa$ such that $X\subset\Sigma{{}\sp\kappa{I}}$. Let us define $\unifU_\kappa$ a special base for the (unique) uniformity on ${}\sp\kappa{I}$. Since $X$ has a unique uniformity, the restriction of $\unifU_\kappa$ to $X$ will define the topology of $\unifpow{X}$. 

Given $F\in[\kappa]\sp{<\omega}$ and $n<\omega$, let 
$$
U_\kappa(F,n)=\{\pair{f,g}:f,g\in{}\sp{\kappa}I,\forall \alpha\in F\ (|f(\alpha)-g(\alpha)|<1/2\sp{n})\}.
$$
Notice that $\unifU_\kappa=\{U_\kappa(F,n):F\in[\kappa]\sp{<\omega},n<\omega\}$ is a family of symmetric open neighborhoods of the diagonal in ${}\sp{\kappa}I\times{}\sp{\kappa}I$ with the finite intersection property. Moreover, $\bigcap\unifU_\kappa$ equals the diagonal. By compactness, it easily follows that $\unifU_\kappa$ is a base of the unique uniformity of ${}\sp{\kappa}I$. 

Given $Y\subset{}\sp\kappa{I}$, we will denote $U_Y(F,n)=U_\kappa(F,n)\cap(Y\times Y)$ and $\unifU_Y=\{U_Y(F,n):F\in[\kappa]\sp{<\omega},n<\omega\}$.

Now, let's play the proximal game on $\unifpow{X}$, we have to give a winning strategy for player 1. Assume that in inning $n<\omega$, player 2 chooses $f_n\in\unifpow{X}$. For each $n<\omega$, let $F_n$ be the set of the first $n$ elements of $\supp{f_i(j)}$ for all $i,j<n$, this is a finite set and is known in inning $n$ before player 1 makes a choice. Define $U_n=U_X(F_n,n)$ for all $n<\omega$. So in inning $n$, make player 1 choose $\widetilde{U_n}$. We claim that this is a winning strategy.

Let $A=\bigcup\{F_n:n<\omega\}$, this is a countable set. Let $P=\{x\in{}\sp\kappa{I}:\supp{x}\subset A\}$, this set is in fact homeomorphic to ${}\sp{A}I$ and contained in $\Sigma{{}\sp\kappa{I}}$. Notice that both $\unifpow{X}$ and $\unifpow{P}$ have the subspace topology of $\unifpow{{}\sp\kappa{I}}$ and also $\unifpow{(X\cap P)}$ is a closed subspace of both $\unifpow{X}$ and $\unifpow{P}$. Moreover, $f_n\in\unifpow{(X\cap P)}$ for each $n<\omega$. So it is enough to prove that the sequence $\{f_n:n<\omega\}$ converges in $\unifpow{P}$. 

Define $V_n=U_P(F_n,n)$ for all $n<\omega$. By the same argument used to prove that $\unifU_\kappa$ is a base for the uniformity of ${}\sp\kappa{I}$, it is possible to prove that in fact the sets $\{V_n:n<\omega\}$ form a base for the uniformity of $P$.

Let us first prove that $\{f_n:n<\omega\}$ converges pointwise, so fix $i<\omega$. Recall that according to the definition of the proximal game, $\pair{f_k(i),f_{k+1}(i)}\in U_k$ for all $k<\omega$. Given $m<n<\omega$ and $\alpha\in F_m$, it easily follows that 
$$
|f_{m+1}(i)(\alpha)-f_n(i)(\alpha)|\leq\sum_{k=m+1}\sp{n-1}{|f_k(i)(\alpha)-f_{k+1}(i)(\alpha)|}<\sum_{k=m+1}\sp{n-1}{1/2\sp{k}}<1/2\sp{m}.
$$
Since the topology of $P$ is that of pointwise convergence and $\{f_n(i)(\alpha):n<\omega\}$ is a Cauchy sequence for all $\alpha\in A$, there exists $f(i)\in P$ to which $\{f_n(i):n<\omega\}$ converges. Notice that $|f_{n+2}(i)(\alpha)-f(i)(\alpha)|\leq 1/2\sp{n+1}<1/2\sp{n}$ every time $n<\omega$ and $\alpha\in F_{n}$. This implies that if $m+2\leq n<\omega$, then $\pair{f(i),f_{n}(i)}\in V_m$.

Thus, we have defined a function $f\in\unifpow{P}$ such that $f_n\in\widetilde{V_m}[f]$ every time that $m+2\leq n<\omega$. Since $\{\widetilde{V_n}[f]:n<\omega\}$ is a local base in $\unifpow{P}$ at $f$, $f$ is the limit of the sequence $\{f_n:n<\omega\}$ in $\unifpow{P}$. By the discussion above, $f$ is the limit of the sequence $\{f_n:n<\omega\}$ in $\unifpow{X}$. Thus, player 1 wins the game, showing that the strategy given is a winning strategy. We have thus finished the proof.

\section{First countability}\label{sectfirstcount}

In this section we give a characterization of first countability for uniform powers. Recall that first countability of a Tychonoff product has a very simple characterization: the product must be countable and each factor space must be first countable. For the box product, not even $\square{}\sp{\omega}(\omega+1)$ is Fr\'echet-Urysohn. However, the countable uniform power of a metric space is metric, so our characterization must include metric spaces.

\begin{lemma}\label{1stcount1}
If $X$ is a separable compactum and $\unifpow{X}$ is first countable, then $X$ is metrizable.
\end{lemma}
\begin{proof}
Let $f\in\unifpow{X}$ be an enumeration of a countable dense subset of $X$ and assume that there is a countable collection $\{\U_n:n<\omega\}$ of open covers of $X$ such that $\{\U_n[f]:n<\omega\}$ forms a local base at $f$.

Assume that $X$ is not metrizable, this means that the covers $\{\U_n:n<\omega\}$ do not generate the uniformity of $X$. Since $X$ is compact, we may assume that conditions (a) and (b) in Lemma \ref{covers} hold.  Thus, there are $x_0,x_1\in X$ such that $x_0\neq x_1$ and no $\U_n$ separates between $x_0$ and $x_1$. Let $V_0,V_1$ be open sets such that $x_0\in V_0\setminus\cl{V_1}$, $x_1\in V_1\setminus\cl{V_0}$ and $X=V_0\cup V_1$.

Define $W=\prod\{W_n:n<\omega\}$, where $W_n=V_0$ if $f(n)\notin V_1$ and $W_n=V_1$ otherwise, for each $n<\omega$. Then $W$ is an open neighborhood of $f$ in $\unifpow{X}$. Let $n<\omega$, we will prove that $\prod\{\Star{f(i),\U_n}:i<\omega\}\not\subset W$.

Define $g\in\prod\{\Star{f(i),\U_n}:i<\omega\}$ in the following way. If $\Star{f(i),\U_n}\cap\{x,y\}\neq\emptyset$, let $g(i)=f(i)$. Otherwise, $\{x_0,x_1\}\subset\Star{f(i),\U_n}$ so let $g(i)=x_1$ if $f(i)\notin V_1$ and $g(i)=x_0$ if $f(i)\in V_1$. The fact that $f$ has a dense image shows that the second case holds for some $i<\omega$ and thus $g\notin W$. This contradiction shows that $X$ must be metrizable.
\end{proof}

\begin{lemma}\label{1stcount2}
If $X$ is a compactum and $\unifpow{X}$ is first countable, then every separable and closed subset of $X$ is of type $G_\delta$.
\end{lemma}
\begin{proof}
Let $A$ be a separable closed subset of $X$ that is not of type $G_\delta$ and let $f\in\unifpow{X}$ be an enumeration of a countable dense subset of $A$. Assume that there is a countable collection $\{\U_n:n<\omega\}$ of open covers of $X$ such that $\{\U[f]:n<\omega\}$ forms a local base at $f$.

Let $U_n=\bigcup\{\Star{f(i),\U_n}:i<\omega\}$, this is an open set and it is easy to see that it contains $A$, for each $n<\omega$. Since $A$ is not of type $G_\delta$, there exists $y\in(\bigcap\{U_n:n<\omega\})\setminus A$. Let $V$ be an open subset of $X$ such that $A\subset V$ and $y\notin V$. Notice that ${}\sp{\omega}V$ is an open neighborhood of $f$.

For $n<\omega$, define $g_n\in\unifpow{X}$ in the following way: if $y\in\Star{f(i),\U_n}$, let $g_n(i)=y$, otherwise, let $g_n(i)=f(i)$. Clearly, $g_n\in\U_n[f]\setminus {}\sp{\omega}V$. This is a contradiction to the fact that $\{\U_n[f]:n<\omega\}$ forms a local base at $f$. Then $X$ cannot be of type $G_\delta$.
\end{proof}

\begin{thm}\label{firstcount}
Let $X$ be a compactum. Then $\unifpow{X}$ is first countable if and only if $X$ is $\omega$-monolithic and every separable closed subspace of $X$ is a $G_\delta$.
\end{thm}
\begin{proof}
First, if $X$ is first countable, then by Lemmas \ref{1stcount1} and \ref{1stcount2}, we are done. So assume that $X$ is $\omega$-monolithic and every closed separable subset of $X$ is of type $G_\delta$.

Let $f\in\unifpow{X}$, we have to describe a countable base at $f$. Call $A=\cl{f[\omega]}$. Since $A$ is of type $G_\delta$, we may choose a sequence $\{U_n:n<\omega\}$ of open subsets of $X$ with intersection $A$ such that $\cl{U_{n+1}}\subset U_n$ for all $n<\omega$. Since $A$ is second countable, by Lemma \ref{covers}, there is a sequence of open covers $\{\U_n:n<\omega\}$ of $X$ such that 
\begin{itemize}
\item[(a)] $\U_{n+1}$ star-refines $\U_n$ for each $n<\omega$
\item[(b)] for every $x,y\in A$ with $x\neq y$, there is $n<\omega$ such that no member of $\U_n$ contains both $x$ and $y$.
\end{itemize}
We may additionally assume that
\begin{itemize}
\item[(c)] $\U_n$ refines the open cover $\{U_n,X\setminus\cl{U_{n+1}}\}$ for all $n<\omega$.
\end{itemize}
We claim that $\{\U_n[f]:n<\omega\}$ is a local base at $f$. So let $\U$ be any open cover of $X$, we must find $m<\omega$ such that $\U_m[f]\subset \U[f]$. To prove this, we first show the claim $(\ast)$ below.

Let $V$ be an open set such that $V\cap A\neq\emptyset$. Define $W(V,n)=\Star{\cl{V}\cap A,\U_n}$. Then $W(V,n)$ is an open set and
$$
(\star)\ \cl{V}\cap A=\bigcap\{W(V,n):n<\omega\}.
$$

Let us prove equation $(\star)$. Clearly the left side is contained in the right side. Now let $x\in X\setminus (\cl{V}\cap A)$. If $x\notin A$, then by (c) there is $n<\omega$ such that $x\notin U_n$ which implies that $x\notin W(V,n)$. If $x\in A$, by property (b) and compactness there is $n<\omega$ such that $\Star{x,\U_n}\cap\cl{V}\cap A=\emptyset$ so $x\notin W(V,n)$.

Now, consider an open cover $\{U\sp\prime:U\in\U\}$ such that $\cl{U\sp\prime}\cap A\subset U$ for each $U\in\U$. Since $\U$ is finite, by $(\star)$ it is possible to find $m<\omega$ such that for every $U\in\U$ with $U\sp\prime\cap A\neq\emptyset$ we have that $W(U\sp\prime,m)\subset U$. We claim that $\U_m[f]\subset\U[f]$.

Let $n<\omega$. Then $f(n)\in A$ and consider some $V\in\U_m$ such that $f(n)\in V$. Let $U\in\U$ be such that $f(n)\in\cl{U\sp\prime}\subset U$. Then it follows that $V\subset W(U\sp\prime,m)$. By the choice of $m$ this implies that $V\subset U$. So this in fact proves that $\Star{f(n),\U_m}\subset\Star{f(n),\U}$. Then $\U_m[f]\subset\U[f]$ and we have finished the proof of this result.
\end{proof}

Todorcevic \cite{todorcevic-handbook} defines an Aronszajn continuum to be a linearly ordered continuum that is first countable, not separable and $\omega$-monolithic. From this definition it immediately follows that every closed and separable subspace of an  Aronszajn continuum is of type $G_\delta$. According to \cite[Proposition 3.6]{todorcevic-handbook}, a linearly ordered continuum is an Aronszajn continuum if it is the Dedekind completion of a densely ordered Aronszajn line. Thus, Aronszajn continua exist in ZFC and provide a non-metrizable example with its uniform power first countable.

\begin{ex}
Any Aronszajn continuum $X$ is a non-metrizable space such that $\unifpow{X}$ is first countable.
\end{ex}

\section{On $\omega_1$}\label{sectomega1}

Notice that the proof of Theorem \ref{main} works for any subspace of a $\Sigma$-product, provided that we use the uniformity defined by the power ${}\sp\kappa{I}$. Notice that $\omega_1$ can in fact be embedded in $\Sigma{{}\sp{\omega_1}{I}}$ by sending $\alpha<\omega_1$ to the function equal to $1$ for all $\beta<\alpha$ and $0$ when $\alpha\leq\beta<\omega_1$. Further, $\omega_1$ has a unique uniformity (because $\omega_1+1$ is its only compactification). Thus, the following result is in fact a corollary of our proof of Theorem \ref{main}. 

\begin{coro} (J. Hankins)
$\unifpow{\omega_1}$ is proximal.
\end{coro}

However, since $\omega_1+1$ is not Fr\'echet-Urysohn, our results cannot tell us whether the countable uniform power of this space is normal.

\begin{ques} (J. R. Bell, \cite{bell-proximal})
Is $\unifpow{(\omega_1+1)}$ normal?
\end{ques}

In this section, we would like to study another special property of $\unifpow{(\omega_1+1)}$, namely, that it is the maximal realcompactification of $\unifpow{\omega_1}$, see Corollary \ref{realcompactification} bellow.

\begin{propo}\label{extending}
$\unifpow{\omega_1}$ is $C$-embedded in $\unifpow{(\omega_1+1)}$.
\end{propo}
\begin{proof}
Let $F:\unifpow{\omega_1}\to\R$ be a continuous function. Let $g\in\unifpow{(\omega_1+1)}\setminus\unifpow{\omega_1}$, we will show that $F$ can be continuously extended to $g$.

Let $I=\{n<\omega:g(n)=\omega_1\}$ and $J=\omega\setminus I$. Let $\beta<\omega_1$ be greater than $\sup\{g(n):n\in J\}$. Notice that $\prod_u\sp{I}{\beta}$ is a closed subspace of $\prod_u\sp{I}{\omega_1}$ so it is proximal. In particular, $\prod_u\sp{I}{\beta}$ is a $W$-space.

Let us find the value that we will assign to $g$ under the extension. For each $\alpha<\omega_1$, let $g_\alpha\in\unifpow{\omega_1}$ be such that ${g\!\!\restriction_{I}}\subset g_\alpha$ and $g_\alpha(t)=\alpha$ for $t\in J$. Then mapping $\alpha$ to $g_\alpha$ is an embedding from $\omega_1$ to $\unifpow{\omega_1}$. Thus, $F$ is eventually constant in this copy of $\omega_1$. Then we may assume that $r\in\R$ is such that $F(g_\alpha)=r$ for all $\alpha>\beta$.

Define $\hat{F}=F\cup\{\pair{g,r}\}$. We have to show that $\hat{F}$ is continuous, so assume that it is not and we will reach a contradiction. Thus, there is $m<\omega$ such that for every neighborhood $W$ of $g$ there is $h\in(\unifpow{\omega_1})\cap W$ with $|F(h)-r|\geq\frac{1}{m+1}$. We will construct a sequence of such $\{h_n:n<\omega\}$ and reach a contradiction.

Start playing the game $\Wgame{\prod_u\sp{I}{\beta},g\!\!\restriction_{I}}$ using the winning strategy of player 1. We will play as player 2. In inning $0$, player 1 chooses an open set $U_0\in\prod_u\sp{I}{\beta}$, choose any $h_0\in\unifpow{\omega_1}$ such that $h_0\!\!\restriction_I\in U_0$ and make player 2 play $h_0\!\!\restriction_I$. In inning $n+1<\omega$, player 1 chooses an open set $U_{n+1}$ in $\prod_u\sp{I}{\beta}$. Let $\beta_{n+1}=\sup\{h_n(i):i\in J\}+1$. The set 
$$
W_{n+1}=\{f\in\unifpow{(\omega_1+1)}:f\!\!\restriction_{I}\in U_{n+1}\textrm{ and if }i\in J, f(i)>\beta_{n+1}\}
$$
is easily seen to be a neighborhood of $g$ so choose $h_{n+1}\in(\unifpow{\omega_1})\cap W_{n+1}$ such that $|F(h_{n+1})-r|\geq\frac{1}{m+1}$. Then make player 2 play $h_{n+1}\!\!\restriction_I$ in the game. By the winning strategy of player 1 we obtain that $\{h_n\!\!\restriction_{I}:n<\omega\}$ converges to $g\!\!\restriction_{I}$. Also, if $i\in J$, from the construction it follows that $h_n(i)<\beta_{n+1}<h_{n+1}(i)$ for all $n<\omega$. Let $\beta\sp\prime=\sup\{\beta_n:1\leq n<\omega\}$. Then it is not hard to see that $\{h_n:n<\omega\}$ converges to $g_{\beta\sp\prime}$. By the continuity of $F$, we obtain that $|F(g_{\beta\sp\prime})-r|>0$. This is a contradiction.

Thus, it follows that $\hat{F}$ is a continuous extension of $F$. Since $g$ was arbitrary, it easily follows that $F$ can be continuously extended to all of $\unifpow{(\omega_1+1)}$.
\end{proof}

Notice that $\unifpow{X}$ is never compact if $X$ is compact and $|X|>2$: let $p\neq q$ be points in $X$ and define $f_n\in\unifpow{X}$ be such that $f_n(k)=p$ if $k<n$ and $f_n(k)=q$ if $k\geq n$, then $\{f_n:n<\omega\}$ converges pointwise to the constant function with value $p$ but does not converge uniformly. However, it turns out that realcompactness is preserved.

Recall that a space is realcompact provided that every ultrafilter of zero sets with the countable intersection property is fixed (\cite[3.11.11]{eng}, \cite[Chapter 8]{gill-jer}). Also, recall that a filter of zero sets $\U$ in a topological space $X$ is called \emph{prime} (\cite[1.44]{walker}) if every time $Z_0$ and $Z_1$ are zero sets of $X$ with $Z_0\cup Z_1\in\U$ then there is $i\in 2$ such that $Z_i\in\U$.

\begin{propo}\label{realcompact}
If $X$ is a compactum, then $\unifpow{X}$ is realcompact.
\end{propo}
\begin{proof}
Let $\U$ be an ultrafilter of zero sets of $\unifpow{X}$ with the countable intersection property, we shall prove that $\U$ is fixed.

For each $n<\omega$, let $\pi_n:\unifpow{X}\to X$ be the projection onto the $n$-th coordinate. Since the uniform power topology contains the Tychonoff topology, it follows that $\pi_n$ is a continuous function for all $n<\omega$. Since the intersection of countably many zero sets is a zero set, we have that if $\{Z_n:n<\omega\}$ are zero sets of $X$, then $\prod\{Z_n:n<\omega\}$ is a zero set of $\unifpow{X}$.

Fix $n<\omega$ for a moment. Let $\U(n)=\{Z:Z\textrm{ is a zero set of }X, \pi_n\sp\leftarrow[Z]\in\U\}$. Then $\U(n)$ is a filter of zero sets in $X$ and since $X$ is compact, $\U(n)$ has a cluster point. From \cite[Propositions 1.44 and 1.45]{walker}, it turns out that $\U(n)$ is a prime filter of zero sets of $X$ and converges to some point which we may call $p_n$. Further, notice that if $Z$ is a zero set of $X$, then since $\U(n)$ is prime, $p_n\in Z$ if and only if $\pi_n\sp\leftarrow[Z]\in\U$.

Let $p\in\unifpow{X}$ be such that $p(n)=p_n$ for all $n<\omega$. We will now prove that $\U$ clusters at $p$. Given $n<\omega$, if we choose a zero-set $Z_n$ such that $p_n\in Z_n$, then by the definition of $\U(n)$ we have that $\pi_n\sp\leftarrow[Z_n]\in\U$. From the fact that $\U$ has the countable intersection property it follows that $p\in\prod\{Z_n:n<\omega\}\in\U$.

Now assume that there is $W\in\U$ such that $p\notin W$. Then there is an open cover $\V$ of $X$ such that $\V[p]\cap W=\emptyset$. For each $n<\omega$, let $f_n:X\to\R$ be a continuous function such that $f_n(p_n)=0$ and $f_n\sp\leftarrow[[0,1)]\subset\Star{p_n,\V}$. Then $Z=\prod\{f_n\sp\leftarrow[[0,1/2]]:n<\omega\}$ is a zero set such that $p\in Z\subset\V[p]$. Thus, $Z\in\U$ and $Z\cap W=\emptyset$, which is a contradiction. This completes the proof that $p$ is a cluster point of $\U$ so $\U$ is fixed.
\end{proof}

Then the next follows by \cite[Theorem 8.7]{gill-jer}.

\begin{coro}\label{realcompactification}
$\unifpow{(\omega_1+1)}$ is the Hewitt-Nachbin realcompactification of $\unifpow{\omega_1}$.
\end{coro}

\section{Semiproximal spaces and products of ordinals}\label{sectprod}

In \cite{nyikos-proximal_semiproximal}, P. Nyikos has defined a uniform space to be semi-proximal if player 2 does not have a winning strategy in the proximal game. In that paper, Nyikos asked whether some properties of proximal spaces are also possesed by semi-proximal spaces.

We are specially interested in the question of whether semi-proximal uniform spaces are normal (\cite[Problem 13]{nyikos-proximal_semiproximal}). We show that products of subspaces of ordinals give a negative answer to this question. This example does not answer the natural modification of Nyikos's problem whether there is an example of a nonnormal topological space that is semi proximal with respect to all compatible uniformities. While products of subspaces of ordinals don't seem to provide an example, it is possible that subspaces of products of copies of $\omega_1$ or related spaces could yield a counterexample, e.g., the spaces constructed in \cite{ladder_systems}.

Let us note that if $X$ is a stationary subset of $\omega_1$ with the induced subspace topology, then for any compatible uniformity $\unifU$ and any $U\in\unifU$ there is an $\alpha<\omega_1$ such that $(X\setminus\alpha)\sp2\subseteq U$.

\begin{ques}
Is it true that every topological space that is semi-proximal with respect to every compatible uniformity is normal?
\end{ques}

Let us start by figuring out when a subspace of $\omega_1$ is proximal.

\begin{thm}\label{subspaces}
Let $X\subset\omega_1$.
\begin{itemize}
\item[(a)] If $X$ is non-stationary, then $X$ is proximal.
\item[(b)] Every subspace of $\omega_1$ is semi-proximal.
\item[(c)] If $X$ is a club, then $X$ is proximal.
\item[(d)] If $X$ and $\omega_1\setminus X$ are stationary, then $X$ is not proximal.
\end{itemize}
\end{thm}
\begin{proof}
If $X$ is non-stationary, then there is a club $C$ such that $X\subset\omega_1\setminus C$. Now, $\omega_1\setminus C$ is then a metric space and every metric space is proximal (\cite[Lemma 4.1]{bell-proximal}). Thus, $X$ is also proximal which proves (a).

Now let us prove (b). Let us consider the uniformity induced on $X$ as a subspace of $\omega_1+1$. Since $\omega_1+1$ is compact and $0$-dimensional, the uniformity of $X$ is defined by clopen partitions. So any entourage in this uniformity is given by a clopen partition of $X$. In particular, we may say that player 1 chooses finite subsets $\{\alpha_0,\ldots,\alpha_n\}\subset\omega_1$, where $0<\alpha_0<\alpha_1<\ldots<\alpha_{n-1}<\alpha_n<\omega_1$, that in turn defines a partition.

So assume that player 2 has a strategy in the proximal game for $X$, we will show how to defeat it. Let $\model$ be a countable elementary submodel of a large enough portion of the universe, containing the strategy of player 2. We may assume that $\alpha=\model\cap\omega_1\in X$. Let $\{\alpha_n:n<\omega\}$ be an increasing sequence of ordinals with limit $\alpha$ and let $F_n=\{\alpha_k:k\leq n\}$ for each $n<\omega$. We will call $x_n\in X$ the point chosen by player 2 in inning $n<\omega$. In inning $n=0,1$, make player 1 choose $F_n$. In inning $n+2$, we have two choices.

If $x_{n+1}<\alpha_n$, notice that in future innings, player 2 is forced to play inside $\alpha_n$. Since $\alpha_n+1$ is a compact metric space, it has a unique uniformity in which player 1 has a winning strategy in the proximal game for $\alpha_n$ (\cite[Lemma 4]{bell-proximal}). Since $X\cap(\alpha_n+1)$ has the uniformity induced by the unique uniformity of $\alpha$, player 1 has a winning strategy in $X\cap(\alpha_n+1)$ so from that inning on, player 1 may start applying that strategy and defeat the strategy of player 2.

Otherwise, $\alpha_n\leq x_{n+1}$ so have player 1 choose $F_{n+2}$. So now assume that this is the case for all $n<\omega$ (that is, the first case considered in the paragraph above never happens). By the choice of $\alpha$ and the fact that $F_n\in\model$ for all $n<\omega$, player 2 is forced to play inside $\alpha$ and in fact $\{x_n:n<\omega\}$ converges to $\alpha$. Thus, we have defeated the strategy of player 2 in this case too. This proves that $X$ is semi-proximal whenever $X$ is stationary. This together with (a) imply (b).

The proof of (c) is similar to the proof of (b) but without the need of elementary submodels. Let $x_n\in X$ be the point chosen by player 2 in inning $n<\omega$. Let $\alpha_0<\omega_1$ be arbitrary; player 1 chooses $\alpha_0$ in inning 0. For $n<\omega$, choose $\alpha_{n+1}<\omega_1$ with $x_n<\alpha_{n+1}$. If there is $n<\omega$ such that $x_{n+1}<\alpha_n$, then from that inning on, player 2 is forced to play inside the compact metric space $X\cap(\alpha_n+1)$ and there is a winning strategy for player 1 (\cite[Lemma 4]{bell-proximal}). Otherwise, player 1 chooses $\{\alpha_k:k\leq n\}$ in inning $n<\omega$; if this is always the case $\{x_n:n<\omega\}$ is forced to be an increasing sequence so it converges in $X$. Notice that $X$ is homeomorphic to $\omega_1$ so this is basically the proof that $\omega_1$ is proximal.

Now we prove (d). Notice that this time we have to consider an arbitrary uniformity $\unifU$ for $X$ and prove that any given strategy for player 1 can be defeated. Let $\model$ be a countable elementary submodel of a large enough portion of the universe, containing the strategy of player 1. We may assume that $\alpha=\model\cap\omega_1\in\omega_1\setminus X$. Moreover, since $X$ is stationary, for each $U\in\unifU$ there exists $\alpha(U)<\omega_1$ such that $(X\setminus\alpha(U))\sp2\subset U$; by elementarity we may assume that $\alpha(U)<\alpha$. Let $\{\alpha_n:n<\omega\}$ be a strictly increasing sequence of ordinals converging to $\alpha$. Then the strategy for player 2 in inning $n<\omega$ is to choose $x_n$ such that $\alpha_n<x_n<\omega_1$ and $\alpha(U_n)<x_n$, where $U_n\in\unifU$ is the element of the uniformity chosen in inning $n$ by player 1. Clearly, $\{x_n:n<\omega\}$ converges to $\alpha\notin X$ and $\emptyset\neq(X\setminus\alpha)\sp2\subset\bigcap\{U_n[x_n]:n<\omega\}$. This shows that this defeats the strategy of player 1.
\end{proof}

In \cite{prods_spaces_ordinals}, normality-type properties of products of subspaces of ordinals were studied. In Corollary 3.3 of that paper, the authors prove that a product $A\times B$, where $A,B\subset\omega_1$, is normal if and only if either one of $A$ or $B$ is not stationary or $A\cap B$ is stationary. We will add another property to their list in Corollary \ref{normalsemiproximal}.

If $\pair{X,\unifU_X}$ and $\pair{Y,\unifU}$ are uniform spaces then the sets of the form
$$
[U,V]=\{\pair{\pair{x_0,y_0},\pair{x_1,y_1}}:x_0,x_1\in X; y_0,y_1\in Y; \pair{x_0,x_1}\in U; \pair{y_0,y_1}\in V\},
$$
where $U\in\unifU_X$ and $V\in\unifU_Y$, form a base for the product uniformity of $X\times Y$.

\begin{thm}\label{semiproxtimesprox}
Let $X$ be a proximal uniform space and let $Y$ be a semi-proximal uniform space. Then $X\times Y$ is semi-proximal with the product uniformity.
\end{thm}
\begin{proof}
Let $\sigma_X$ denote the winning strategy for player 1 in the proximal game for $X$. Assume that player 2 has a strategy $\sigma_{X\times Y}$ for $X\times Y$, we will see that this strategy cannot be a winning strategy.

Let us first describe a strategy $\sigma_Y$ for player 2 in the proximal game for $Y$. We know that $\sigma_Y$ can be defeated and this will help us defeat $\sigma_{X\times Y}$. We will be playing three proximal games simultaneously: for $X$, $X\times Y$ and $Y$. Thus, we will denote proximal game for a space $Z$ by $\Gamma_Z$.

Call $U_0$ the entourage given by $\sigma_X$ in inning $0$ of $\Gamma_X$. Let $V_0$ be the entourage of $Y$ given by player 1 in $\Gamma_Y$. Then, make player 1 in $\Gamma_{X\times Y}$ play $[U_0,V_0]$. Using $\sigma_\unifU$ we obtain $\pair{x_0,y_0}\in X\times Y$, the response of player 2 in $\Gamma_{X\times Y}$. Then, in $\Gamma_Y$, we define the strategy $\sigma_Y$ in this situation so that player 2 chooses $y_{n+1}$.

Assume now that $n<\omega$, we are in inning $n+1$ of $\Gamma_Y$, player 1 has chosen entourages $V_0,\ldots,V_n$ of $Y$ and the $\sigma_Y$ has been defined in such a way that player 2 has responded with $y_0,y_1,\ldots,y_n$ in $Y$. Additionally, there exist entourages $U_0,\ldots,U_n$ of $X$ so that: (a) $\Gamma_X$ has been played using strategy $\sigma_X$, player 1 has played $U_0,\ldots,U_n$ and player 2 has played $x_0,\ldots,x_n$; (b) $\Gamma_{X\times Y}$ has been played using strategy $\sigma_{X\times Y}$, player 1 has played $[U_0,V_0],\ldots,[U_n,V_n]$ and player 2 has played $\pair{x_0,y_0},\ldots,\pair{x_n,y_n}$.

According to $\sigma_X$, player 1 chooses an entourage $U_n$ in inning $n+1$ of $\Gamma_X$. Now assume that player 1 chooses $V_{n+1}$ in inning $n+1$ of $\Gamma_Y$. Make player 1 in $\Gamma_{X\times Y}$ play $[U_{n+1},V_{n+1}]$ in inning $n+1$. So according to $\sigma_\unifU$, player 2 responds $\pair{x_{n+1},y_{n+1}}$ in $\Gamma_{X\times Y}$. So we define $\sigma_Y$ in this situation by making player 2 choose $y_{n+1}$.

This completely defines $\sigma_Y$. In the following, when we are refering to the sets given by $\sigma_Y$, we will use the same notation as in the definition of $\sigma_Y$ above.

Since $\sigma_Y$ is an strategy for player 2 in $\Gamma_Y$, there is a way to defeat it. This means that there exists $\{V_n:n<\omega\}$ such that either $\{y_n:n<\omega\}$ converges or $\bigcap\{V_n[y_n]:n<\omega\}=\emptyset$. Moreover, since $\sigma_X$ is a winning strategy for player 1 in $\Gamma_X$, then either $\{x_n:n<\omega\}$ converges or $\bigcap\{U_n[x_n]:n<\omega\}=\emptyset$. From this it follows that either $\{\pair{x_n,y_n}:n<\omega\}$ converges or $\bigcap\{[U_n,V_n][\pair{x_n,y_n}]:n<\omega\}=\emptyset$. Thus, we have defeated $\sigma_{X\times Y}$. So $X\times Y$ is semi-proximal.
\end{proof}

In order to analyze semi-proximality in products of stationary subsets of $\omega_1$, we will use the following result. This simply says that any uniformity in such a product has an unbounded section.

\begin{lemma}\label{productunbounded}
Let $A,B\subset\omega_1$ be stationary and let $\unifU$ be a uniformity defined on $A\times B$. Then for each $U\in\unifU$ there exists $\eta<\omega_1$ and a stationary $T\subset\omega_1\setminus\eta$ such that for each $\beta\in T$ there is $\gamma_\beta<\omega_1$ with
$$
((\gamma_\beta,\omega_1)\times(\eta,\beta])\sp2\subset U.
$$
\end{lemma}
\begin{proof}
Given $U\in\unifU$, consider an entourage $V\in\unifU$ such that $V\circ V\subset U$. For each $p=\pair{\alpha,\beta}\in A\times B$, let $\gamma_p<\alpha$ and $\eta_p<\beta$ such that 
$$
(\ast)\ [(\alpha+1\setminus\gamma_p)\times(\beta+1\setminus\eta_p)]\cap (A\times B)\subset V[p].
$$

Fix $\beta<\omega_1$ for the moment. Since $A$ is stationary, there is $S_\beta\in[A]\sp{\omega_1}$ and $\gamma_\beta<\omega_1$ such that $\gamma_p=\gamma_\beta$ for all $p\in A\times\{\beta\}$. Moreover, we may assume that there is $\eta_\beta<\omega_1$ such that $\eta_p=\eta_\beta$ for all $p\in A\times\{\beta\}$.

Now, since $B$ is stationary, there exists $T\in[B]\sp{\omega_1}$ and $\eta<\omega_1$ such that $\eta_\beta=\eta$ for all $\beta\in T$. Now we prove that these choices are what we are looking for.

Let $x,y\in(\gamma_\beta,\omega_1)\times(\eta,\beta)$ for some $\beta\in T$. Then there is $\alpha\in S_\beta$ such that $x,y\in(\gamma_\beta,\alpha+1)\times(\eta,\beta)$. Define $p=\pair{\alpha,\beta}$. By equation $(\ast)$ above it follows that $x,y\in V[p]$. Thus, $\pair{x,p},\pair{p,y}\in V$ and this implies that $\pair{x,y}\in U$. This concludes the proof of this lemma.
\end{proof}

\begin{thm}\label{prods}
Let $A,B\subset\omega_1$ and consider them with any compatible uniformity. Then $A\cap B$ is stationary if and only if $A\times B$ is semi-proximal.
\end{thm}
\begin{proof}
First, assume that $A\cap B$ is stationary, we have to prove that $A\times B$ is semi-proximal. The proof of this is similar to the proof of (b) Theorem \ref{subspaces}. So consider $A\times B$ with the product uniformity. Assume that player 2 has a winning strategy in the proximal game and let $\model$ be a countable elementary submodel of the universe that contains this strategy and $A,B\in\model$. We may assume that $\alpha=\model\cap\omega_1\in A\cap B$.

For every $\beta<\omega_1$, consider the following finite clopen partition of $A\times B$
$$
\U_\beta=(\beta\times\beta)\cup((A\setminus\beta)\times\beta)\cup(\beta\times(B\setminus\beta))\cup((A\setminus\beta)\times(B\setminus\beta)).
$$
Also, choose a strictly increasing sequence $\{\alpha_n:n<\omega\}$ that converges to $\alpha$.

In innings $n=0,1$, let $\beta_n=\alpha_n$ and choose $U_n=\bigcup\{U\sp2:U\in\V_{\beta_n}\}$, where $\V_0=\U_{\beta_0}$ and $\V_1$ is the common refinement of $\U_{\beta_0}$ and $\U_{\beta_1}$. Call $\pair{x_n,y_n}\in A\times B$ the point chosen by player 2 in inning $n<\omega$. Assume that for each $n<\omega$, in inning $n+1$ we have chosen $\beta_{n+1}$ such that $\beta_n<\beta_{n+1}<\omega_1$ and player 1 has chosen $U_{n+1}=\bigcup\{U\sp2:U\in\V_{n+1}\}$, where $\V_{n+1}$ is the common refinement of $\U_{\beta_0},\U_{\beta_1},\ldots,\U_{\beta_{n+1}}$. Now let $n<\omega$, we will analyse what to do in inning $n+2$. Notice that $\pair{x_{n+1},y_{n+1}}\in W\sp2$ for some $W\in\V_n$. Moreover, this $W$ is of the form $C\times D$, where $C$ and $D$ are intervals of $A$ and $B$, respectively.

If $W\neq(A\setminus\beta_n)\times(B\setminus\beta_n)$, then the rest of the game will take place inside a rectangle one of whose sides is metric and thus, $W$ is semi-proximal by Theorem \ref{semiproxtimesprox}. So there is a way to defeat the strategy of player 2 in this case.

So assume that $W=(A\setminus\beta_n)\times(B\setminus\beta_n)$. Then let $\beta_{n+2}\in\model$ be such that $\beta_{n+1},x_{n+1},y_{n+1}<\beta_{n+2}$ and make player 1 choose $U_{n+2}=\bigcup\{U\sp2:U\in\V_{n+2}\}$, where $\V_{n+2}$ is the common refinement of $\U_{\beta_0},\U_{\beta_1},\ldots,\U_{\beta_{n+2}}$. If this is the case for all $n<\omega$, then we obtain that in fact $\{\pair{x_n,y_n}:n<\omega\}$ converges to $\pair{\alpha,\alpha}$. Thus, we have defeated the strategy of player 1.

Now assume that $A\cap B$ is not stationary. We now have to show that player 2 has a winning strategy. Let $C\subset\omega_1$ be a club disjoint from $A\cap B$. Notice that $C$ intersects both $A$ and $B$ unboundedly often. We shall define the point $\pair{x_n,y_n}$ chosen by player 2 in inning $n$ in such a way that 
\begin{itemize}
\item[(a)] $\{x_n:n<\omega\}\subset A\cap C$,
\item[(b)] $\{y_n:n<\omega\}\subset B\cap C$ and 
\item[(c)] $y_n<x_n<y_{n+1}$ for every $n<\omega$.
\end{itemize}

In inning $m<\omega$, assume that player 1 has chosen an entourage $U_m$. Also, by induction we will assume that the sequence $\{x_n:n<m\}$ constructed so far so that conditions (a), (b) and (c) above hold for $n<m$. Let $V_m$ be an entourage such that $V_m\circ V_m\subset U_m$. Apply Lemma \ref{productunbounded} to $V_m$ and let $\eta_m<\omega_1$ and $T_m\subset\eta_m$ the sets thus obtained.

Let $\beta_m\in T_m$ be such that $\max{(\{x_n:n<m\}\cup\{y_n:n<m\})}+\omega\leq\beta_m$. Recall that there is $\gamma_{\beta_m}$ with the properties in the statement of Lemma \ref{productunbounded}. Choose $y_m\in B\cap C\cap\beta_m$ such that $\max{(\{x_n:n<m\}\cup\{y_n:n<m\})}+1<y_m$. After this, choose $x_m\in A\cap C$ such that $y_m<x_m$ and $\gamma_{\beta_m}<x_m$. Then in this inning, player 2 should play $\pair{x_m,y_m}$. This completes the definition of the strategy.

It follows that there is $c\in C$ such that $\{\pair{x_n,y_n}:n<\omega\}$ converges to $\pair{c,c}\notin A\times B$. Then we only need to prove that $\bigcap\{U_n[\pair{x_n,y_n}]:n<\omega\}\neq\emptyset$ to show that this strategy is winning for player 2.

For each $n<\omega$, let $\beta\sp\prime_n\in T_n$ be such that $c+\omega\leq\beta\sp\prime_n$ and let $\gamma_{\beta\sp\prime_n}<\omega_1$ be given with the properties in the statement of Lemma \ref{productunbounded}. Let $p\in A$ and $q\in B$ be such that $\gamma_{\beta_n}<p$, $\gamma_{\beta\sp\prime_n}<p$ and $q<\beta\sp\prime_n$ for all $n<\omega$. We claim that $\pair{p,q}\in\bigcap\{U_n[\pair{x_n,y_n}]:n<\omega\}$.

Take $n<\omega$, let us see that $\pair{p,q}\in U_n[\pair{x_n,y_n}]$. Consider the point $\pair{p,y_n}\in [(\gamma_{\beta_n},\omega_1)\times(\eta,\beta_n]]\cap[(\gamma_{\beta\sp\prime_n},\omega_1)\times(\eta,\beta\sp\prime_n]]$. Since both $\pair{x_n,y_n}$ and $\pair{p,y_n}$ are elements of $(\gamma_{\beta_n},\omega_1)\times(\eta,\beta_n]$, by Lemma \ref{productunbounded} we obtain that $\pair{\pair{x_n,y_n},\pair{p,y_n}}\in V_n$. Since both $\pair{p,y_n}$ and $\pair{p,q}$ are elements of $(\gamma_{\beta\sp\prime_n},\omega_1)\times(\eta,\beta\sp\prime_n]$, by Lemma \ref{productunbounded} we obtain that $\pair{\pair{p,y_n},\pair{p,q}}\in V_n$. By the choice of $V_n$ this implies that $\pair{\pair{x_n,y_n},\pair{p,q}}\in U_n$. Thus, $\pair{p,q}\in U_n[\pair{x_n,y_n}]$.

Thus, this strategy makes player 2 win.
\end{proof}

From this and Corollary 3.3 from \cite{prods_spaces_ordinals}, the next follows immediately.

\begin{coro}\label{normalsemiproximal}
If $A,B\subset\omega_1$, then the following conditions are equivalent.
\begin{itemize}
\item[(a)] $A\times B$ is collectionwise normal,
\item[(b)] $A\times B$ is normal,
\item[(c)] either one of $A$ or $B$ is not stationary or $A\cap B$ is stationary, and
\item[(d)] $A\times B$ is semi-proximal.
\end{itemize}
\end{coro}

Notice that we also obtain that the product of two semi-proximal spaces that are Fr\'echet may have their product not semi-proximal; this answers Problem 14 in \cite{nyikos-proximal_semiproximal}.

\begin{coro}
There are two semi-proximal spaces $A,B\subset\omega_1$ such that $A\times B$ is not semi-proximal.
\end{coro}

\section{The topological proximal game}\label{topproxgame}

Now we consider a variation of the proximal game on a topological space where player 1 plays open covers instead of elements of a uniformity. Given a topological space $X$, in inning $n<\omega$ of the game, player 1 plays an open cover $\U_n$ and player 2 a point $x_n\in X$. The choice of player 2 must lie in the star of the previous point with respect to the previous open cover: $x_{n+1}\in\Star{x_n,\U_n}$ for each $n<\omega$. Player 1 wins this instance of the game if either $\{x_n:n<\omega\}$ converges or $\bigcap\{\Star{x_n,\U_n}:n<\omega\}=\emptyset$. Call this game the \emph{topological proximal game}, $\tpgame{X}$ for short. The following follows easily.

\begin{thm}\label{proxvstopprox}
Suposse that $X$ is any topological space. 
\begin{itemize}
\item[(a)] If $\pair{X,\unifU}$ is proximal for some uniformity $\unifU$, then player 1 has a winning strategy in $\tpgame{X}$. 
\item[(b)] If player 2 has a winning strategy in $\tpgame{X}$, then $\pair{X,\unifU}$ is not semi-proximal for any uniformity $\unifU$.
\end{itemize}
\end{thm}

On the other hand, the topological proximal game is not equivalent to the proximal game as at least (b) in Theorem \ref{proxvstopprox} cannot be reversed.

\begin{thm}
If $A$ and $B$ are disjoint stationary subsets of $\omega_1$, then player 2 has a winning strategy in the proximal game with respect to any uniformity on $A\times B$ but player 2 has no winning strategy in $\tpgame{A\times B}$.
\end{thm}
\begin{proof}
In Theorem \ref{prods} we have already seen that player 2 has a winning strategy in the proximal game with respect to any uniformity. To see that there is no winning strategy for player 2 in $\tpgame{A\times B}$, fix a strategy $\sigma$ for player 2. 

Consider the following open cover of $A\times B$
$$
\U_0=\{(0,\alpha]\times(\alpha,\omega_1):\alpha\in A\}\cup\{(\beta,\omega_1)\times(0,\beta]:\beta\in B\}.
$$

In inning $0$, let player 1 choose $\U_0$. Let $x_0=\pair{\alpha,\beta}$ be the choice of player 2 in inning $0$. Assume that $\beta<\alpha$. Then
$$
\Star{x_0,\U_0}=\bigcup\{(\gamma,\omega_1)\times(0,\gamma]:\gamma\in B,\beta\leq\gamma<\alpha+1\},
$$
which is bounded on the second coordinate by $\alpha+1$. Thus, starting from inning $1$ on, the game is played on the subspace $A\times (B\cap(\alpha+1))$. By Theorems \ref{semiproxtimesprox} and \ref{proxvstopprox} we know that player 2 has no winning strategy in $\tpgame{A\times (B\cap(\alpha+1))}$ so we may play in this way and defeat $\sigma$.

If $\alpha<\beta$, a similar argument can be given to defeat $\sigma$. This shows that $\sigma$ cannot be a winning strategy.
\end{proof}

\begin{ques}
If player 1 has a winning strategy in $\tpgame{X}$, does it follow that $X$ is proximal with respect to some (every) compatible uniformity on $X$?
\end{ques}

\begin{thm}
If $X$ is paracompact and player 1 has a winning strategy in $\tpgame{X}$, then $X$ is proximal with respect to some uniformity on $X$.
\end{thm}
\begin{proof}
Since $X$ is paracompact, every open cover has a star-refinement. By Lemma \ref{covers} we obtain that the set of all open covers of $X$ generates a uniformity on $X$. Clearly, $X$ is proximal with respect to this uniformity.
\end{proof}

\end{document}